\documentclass{amsart}
\usepackage{all2,xypic}
\newcommand{\cF}{\mathcal{F}}

\newcommand{\mon}{\operatorname{M}}     
\newcommand{\matr}{\operatorname{H}}     

\newcommand{\ideal}[1]{\langle #1 \rangle}
\renewcommand{\phi}{\varphi}

\begin{document}

\title[Noetherianity for infinite-dimensional toric varieties]{Noetherianity for infinite-dimensional \\toric varieties}
\author[J.~Draisma]{Jan Draisma}
\address[Jan Draisma]{
Department of Mathematics and Computer Science\\
Technische Universiteit Eindhoven\\
P.O. Box 513, 5600 MB Eindhoven, The Netherlands\\
and Vrije Universiteit and Centrum voor Wiskunde en Informatica, Amsterdam,
The Netherlands}
\thanks{The first and second author are supported by a Vidi grant from
the Netherlands Organisation for Scientific Research (NWO)}
\email{j.draisma@tue.nl}

\author[R.H.~Eggermont]{Rob H.~Eggermont}
\address[Rob H.~Eggermont]{
Department of Mathematics and Computer Science\\
Technische Universiteit Eindhoven\\
P.O. Box 513, 5600 MB Eindhoven, The Netherlands}
\email{r.h.eggermont@tue.nl}

\author[R.~Krone]{Robert Krone}
\address[Robert Krone]{
School of Mathematics\\
Georgia Institute of Technology\\
686 Cherry Street\\
Atlanta, GA 30332-0160 USA}
\email{krone@math.gatech.edu}
\thanks{The third and fourth author are supported in part by the NSF under grant DMS~1151297.}

\author[A.~Leykin]{Anton Leykin}
\address[Anton Leykin]{
School of Mathematics\\
Georgia Institute of Technology\\
686 Cherry Street\\
Atlanta, GA 30332-0160 USA}
\email{leykin@math.gatech.edu}

\begin{abstract}
We consider a large class of monomial maps respecting an action of the
infinite symmetric group, and prove that the toric ideals arising as
their kernels are finitely generated up to symmetry. Our class includes
many important examples where Noetherianity was recently proved or
conjectured. In particular, our results imply Hillar-Sullivant's
Independent Set Theorem and settle several finiteness conjectures due
to Aschenbrenner, Martin del Campo, Hillar, and Sullivant.

We introduce a {\em matching monoid} and show that its monoid ring is
Noetherian up to symmetry. Our approach is then to factorize a more
general equivariant monomial map into two parts going through this
monoid. The kernels of both parts are finitely generated up to symmetry:
recent work by Yamaguchi-Ogawa-Takemura on the (generalized) Birkhoff
model provides an explicit degree bound for the kernel of the first part,
while for the second part the finiteness follows from the Noetherianity
of the matching monoid ring.
\end{abstract}

\maketitle

\section{Introduction and main result} \label{sec:Intro}

Families of algebraic varieties parameterized by combinatorial
data arise in various areas of mathematics, such as
statistics (e.g., phylogenetic models parameterized by trees
\cite{Allman04,Draisma07b,Draisma12f,Pachter05} or the relations among
path probabilities in Markov chains parameterized by path length
\cite{Haws12,Noren12}), commutative algebra (e.g., Segre powers of
a fixed vector space parameterized by the exponent \cite{Snowden10}
or Laurent lattice ideals \cite{Hillar13}), and combinatorics (e.g.,
algebraic matroids arising from determinantal ideals parameterized by
matrix sizes \cite{Kiraly13} or edge ideals of hypergraphs parameterized
by the number of vertices~\cite{GrossPetrovic12}).  A natural question
is whether such families {\em stabilize} as some of the combinatorial
data tend to infinity. A recently established technique for proving such
stabilization is passing to an infinite-dimensional limit of the family,
giving some equations for that limit, and showing that those equations
cut out a suitably Noetherian space. This then implies that the limit
itself is given by finitely many further equations, and that the family
stabilizes. This technique is applied, for instance, in the proof of
the Independent Set Theorem \cite{hillar2012finite}, and in the first
author's work on the Gaussian $k$-factor model, chirality varieties,
and tensors of bounded rank \cite{Draisma08b,Draisma11d}.

In the present paper, we follow a similar approach, utilizing the new
concept of a {\em matching monoid} to prove that stabilization happens
for a large class of toric varieties.  Our Main Theorem provides one-step
proofs for several existing results that were established in a rather
less general context; and it settles conjectures and questions from
\cite{Aschenbrenner07,hillar2012finite,Hillar13}.  There is a list of
three such consequences at the end of this Introduction.  Moreover, we
show Noetherianity in a constructive manner by complementing the Main
Theorem with an algorithm that produces a finite set of equations whose
orbits define the infinite-dimensional toric variety under consideration.

Instead of
working with inverse systems of affine varieties, we work directly with
direct limits of their coordinate rings. In fact, we formulate our
Main Theorem directly in the infinite-dimensional setting, as going back
to families of finite-dimensional coordinate rings of toric varieties is
fairly straightforward.  Throughout, $\NN$ denotes $\{0,1,2,3,\ldots\},$
and for $k \in \NN$ we write $[k]:=\{0,\ldots,k-1\}$. We write $\Sym(\NN)$
for the group of all bijections $\NN \to \NN$, and $\Inc(\NN)$ for the
monoid of all strictly increasing maps $\NN \to \NN$. Let $Y$ be a set
equipped with an action of $\Sym(\NN)$. We require that the action has
the following property: for each $y \in Y$ there exists a $k_y \in \NN$
such that $y$ is fixed by all of $\Sym(\NN\setminus [k_y]),$ i.e.,
by all elements of $\Sym(\NN)$ that fix $[k_y]$ element-wise. In this
setting, $\Inc(\NN)$ also acts on $Y,$ as follows: for $\pi \in \Inc(\NN)$
and $y \in Y,$ choose a $\pi' \in \Sym(\NN)$ that agrees with $\pi$ on
$[k_y],$ set $\pi y:=\pi' y,$ and observe that this does not depend on the
choice of $\pi'$.  Observe that for each $y \in Y$ the $\Inc(\NN)$-orbit
$\Inc(\NN)y$ is contained in $\Sym(\NN)y$, and that the latter is in fact
equal to the orbit of $y$ under the countable subgroup of $\Sym(\NN)$
consisting of permutations fixing all but finitely many natural numbers.
See also \cite[Section 5]{hillar2012finite}.

Let $R$
be a Noetherian ring (commutative, with $1$), and let $R[Y]$ be the commutative
$R$-algebra of polynomials in which the elements of $Y$ are the variables
and the coefficients come from $R$. The group $\Sym(\NN)$
acts by $R$-algebra automorphisms on $R[Y]$
by permuting the variables. Furthermore, let $k$ be a natural number, and
let $Z=\{z_{ij} \mid i \in [k], j \in \NN\}$ be a second set of variables,
with a $\Sym(\NN)$-action given by $\pi z_{ij}=z_{i \pi(j)}$. Extend this
action to an action by $R$-algebra automorphisms of $R[Z]$. Note that the $\Sym(\NN)$-actions on $R[Y],Z,R[Z]$
all have the property required of the action on $Y$.
Hence they also yield $\Inc(\NN)$-actions, by means of injective
$R$-algebra endomorphisms in the case of $R[Y]$ and $R[Z]$.
In general, when a monoid $\Pi$ acts on a ring $S$ by means
of endomorphisms, $S$ is called {\em $\Pi$-Noetherian} if every
$\Pi$-stable ideal in $S$ is generated by the union of finitely many
$\Pi$-orbits of elements, i.e., if $S$ is Noetherian as a
module under the {\em skew monoid ring} $S * \Pi$; see
\cite{hillar2012finite}.

\begin{thm}[Main Theorem]\label{thm:main}
Assume that $\Sym(\NN)$ has only finitely many orbits on $Y$. Let
$\phi:R[Y] \to R[Z]$ be a $\Sym(\NN)$-equivariant homomorphism that
maps each $y \in Y$ to a {\em monomial} in the $z_{ij}$. Then $\ker \phi$
is generated by finitely many $\Inc(\NN)$-orbits of binomials, and
$\im \phi \cong R[Y]/\ker \phi$ is an $\Inc(\NN)$-Noetherian ring.
\end{thm}

If an ideal is $\Sym(\NN)$-stable, then it is certainly
$\Inc(\NN)$-stable, so the last statement implies that $R[Y]/\ker \phi$
is $\Sym(\NN)$-Noetherian. The conditions in the theorem are sharp in
the following senses.

\begin{enumerate}
\item The ring $R[Y]$ itself is typically {\em not}
$\Sym(\NN)$-Noetherian, let alone $\Inc(\NN)$-Noetherian.
Take, for instance, $Y=\{y_{ij} \mid i,j \in \NN\}$ with
$\Sym(\NN)$ acting diagonally on both indices, and take
any $R$ with $1 \neq 0$. Then the $\Sym(\NN)$-orbits of the monomials
\[
y_{12}y_{21},y_{12}y_{23}y_{31},y_{12}y_{23}y_{34}y_{41},\ldots
\]
generate a $\Sym(\NN)$-stable ideal that is not generated by any finite
union of orbits (see \cite[Proposition 5.2]{Aschenbrenner07}).

\item The $R$-algebra $R[Z]$ is $\Sym(\NN)$-Noetherian, and even
$\Inc(\NN)$-Noetherian \cite{Cohen87,hillar2012finite}---this is the
special case of our theorem where $Y=Z$ and $\phi$ is the identity---but
$\Sym(\NN)$-stable subalgebras of $R[Z]$ need not be, even when
generated by finitely many $\Sym(\NN)$-orbits of polynomials. For
instance, an (as yet) unpublished theorem due to Krasilnikov says that
in characteristic $2$, the ring generated by
all $2 \times 2$-minors of a $2 \times \NN$-matrix of variables is
not $\Sym(\NN)$-Noetherian. Put differently, we do not know if the
finite generatedness of $\ker \phi$ in the Main Theorem continues to
hold if $\phi$ is an arbitrary $\Sym(\NN)$-equivariant homomorphism,
but certainly the quotient is not, in general, $\Sym(\NN)$-Noetherian.

\item Moreover, subalgebras of $R[Z]$ generated by finitely many
$\Inc(\NN)$-orbits of {\em monomials} need not be $\Inc(\NN)$-Noetherian;
see Krasilnikov's example in \cite{hillar2012finite}.  However, our Main
Theorem implies that subalgebras of $R[Z]$ generated by finitely many
$\Sym(\NN)$-orbits of monomials {\em are} $\Inc(\NN)$-Noetherian.
\end{enumerate}

Our Main Theorem applies to many problems on Markov bases of families
of point sets. In such applications, the following strengthening is
sometimes useful.

\begin{cor}\label{cor:main}
Assume that $\Sym(\NN)$ has only finitely many orbits on
$Y,$ and let $S$ be an $R$-algebra with trivial
$\Sym(\NN)$-action. Let
$\phi:R[Y] \to S[Z]$ be a $\Sym(\NN)$-equivariant
$R$-algebra homomorphism that
maps each $y \in Y$ to an element of $S$ times a monomial in the $z_{ij}$. Then $\ker \phi$
is generated by finitely many $\Inc(\NN)$-orbits of binomials, and
$\im \phi \cong R[Y]/\ker \phi$ is an $\Inc(\NN)$-Noetherian ring.
\end{cor}

\begin{proof}[Proof of the Corollary given the Main
Theorem.]
Let $y_p, p \in [N]$ be representatives of the
$\Sym(\NN)$-orbits on $Y.$ Then for all $p \in [N]$ and
$\pi \in \Sym(\NN)$ we have $\phi(\pi y_p)=
s_p \pi u_p$ for some monomial $u_p$ in the
$z_{ij}$ and some $s_p$ in $S$. Apply the Main Theorem to $Y':=Y \times \NN$ and $Z \cup
Z'$ with $Z':=\{z'_{p,j} \mid p \in [N], j \in \NN\}$ and $\phi'$ the map
that sends the variable $(\pi y_p,j)$ to $z'_{p,j} \pi u_p$.
Consider the commutative diagram
\[
\xymatrix{
R[Y'] \ar[r]^{\phi'} \ar[d]^{\rho: (y,j) \mapsto y} &
	R[Z \cup Z'] \ar[d]^{\psi: z'_{pj} \mapsto s_p} \\
R[Y] \ar[r]^{\phi} & S[Z]
}
\]
of $\Sym(\NN)$-equivariant $R$-algebra homomorphisms. By the Main
Theorem, $\im \phi'$ is $\Inc(\NN)$-Noetherian, hence so is its image
under $\psi$; and this image equals $\im \phi$ because $\rho$ is surjective.
Similarly, $\ker(\psi \circ \phi')$ is generated by finitely many
$\Inc(\NN)$-orbits (because this is the case for both $\ker
\phi'$ and $\ker \psi|_{\im \phi'}$), hence so is its image under $\rho$; and this image is $\ker \phi$ because $\rho$ is surjective.
\end{proof}

Here are some consequences of our Main Theorem.

\begin{enumerate}
\item Our Main Theorem implies \cite[Conjecture 5.10]{Aschenbrenner07}
that chains of ideals arising as kernels of monomial maps of the form
$y_{i_1,\ldots,i_k} \mapsto z_{i_1}^{a_1} \cdots z_{i_k}^{a_k}$, where
the indices $i_1,\ldots,i_k$ are required to be distinct, stabilize.
In \cite{Aschenbrenner07} this is proved in the squarefree case, where
the $a_j$ are equal to $1$. In the Laurent polynomial setting more is
known \cite{Hillar13}.

\item A consequence of \cite{deLoera95} is that for any $n \geq 4$ the
vertex set $\{v_{ij}:=e_i + e_j \mid i \neq j\} \subseteq \RR^n$ of the
$(n-1)$-dimensional second hypersimplex has a Markov basis corresponding
to the relations $v_{ij}=v_{ji}$ and $v_{ij}+v_{kl}=v_{il}+v_{kj}$.
Here is a qualitative generalisation of this fact. Let $m$ and $k$ be
fixed natural numbers.  For every $n \in \NN$ consider a finite set $P_n
\subseteq \ZZ^m \times \ZZ^{k \times n}$. Let $\Sym(n)$ act trivially on
$\ZZ^m$ and by permuting columns on $\ZZ^{k \times n}$.  Assume that there
exists an $n_0$ such that $\Sym(n)P_{n_0}=P_n$ for $n \geq n_0$; here we
think of $\ZZ^{k \times n_0}$ as the subset of $\ZZ^{k \times n}$ where
the last $n-n_0$ columns are zero. Then Corollary~\ref{cor:main} implies
that there exists an $n_1$ such that for any Markov basis $M_{n_1}$
for the relations among the points in $P_{n_1}$, $\Sym(n) M_{n_1}$ is a Markov basis for $P_n$ for all $n \geq
n_1$. For the second hypersimplex, $n_0$ equals $2$ and $n_1$ equals $4$.

\item A special case of the previous consequence is the
Independent Set Theorem of \cite{hillar2012finite}. We briefly
illustrate how to derive it directly from Corollary~\ref{cor:main}.
Let $m$ be a natural number and let $\cF$ be a family of subsets of
a finite set $[m]$. Let $T$ be a subset of $[m]$ and assume
that each $F \in \cF$ contains at most one element of $T$.
In other words, $T$ is an independent set in the hypergraph
determined by $\cF$. For $t \in [m] \setminus T$ let $r_t$
be a natural number. Set $Y:=\{y_\alpha \mid
\alpha \in \NN^T \times \prod_{t \in [m] \setminus T} [r_t] \}$ and
$Z:=\{z_{F,\alpha} \mid F \in \cF, \alpha \in \NN^{F \cap
T} \times \prod_{F \setminus T} [r_t] \},$ and let $\phi$
be the homomorphism $\ZZ[Y] \to \ZZ[Z]$ that maps
$y_{\alpha}$ to $\prod_{F \in \cF} z_{F,\alpha|_F},$ where
$\alpha|_F$ is the restriction of $\alpha$ from $[m]$ to
$F$. Then $\phi$ is equivariant with respect to the action
of $\Sym(\NN)$ on the variables induced by the diagonal
action of $\Sym(\NN)$ on $\NN^T,$ and (a strong form of) the Independent Set
Theorem boils down to the statement that $\ker \phi$ is
generated by finitely many $\Sym(\NN)$-orbits of binomials.
By the condition that $T$ is an independent set, each $z$-variable has at most
one index running through all of $\NN$. Setting $S$ to be
$\ZZ[z_{F,\alpha} \mid F \cap T = \emptyset],$
we find that $Y,$ $S,$ the remaining
$z_{F,\alpha}$-variables, with $|F \cap T|=1,$ and the map
$\phi$ satisfy the conditions of the corollary. The
conclusion of the corollary now implies the Independent Set
Theorem.
\end{enumerate}

The remainder of the paper is organized as follows: In
Section~\ref{sec:Reduction} we reduce the Main Theorem to a particular
class of maps $\phi$ related to {\em matching monoids} of complete
bipartite graphs. For these maps, finite generation of the kernel
follows from recent results on the Birkhoff model \cite{Yamaguchi14}; see
Section~\ref{sec:Relations}, where we also describe the image of $\phi$.
In Section~\ref{sec:Noetherianity1} we prove Noetherianity of $\im \phi,$
still for our special $\phi$. As in \cite{Cohen87,hillar2012finite}, the
strategy in Section~\ref{sec:Noetherianity1} is to prove that a partial
order on certain monoids is a well-partial-order. In our case, these are
said matching monoids, and the proof that they are well-partially-ordered
is quite subtle.  In Section~\ref{sec:Computational} we
establish that a finite $\Inc(\NN)$-generating set of $\ker \phi$ is
(at least theoretically) computable. The last Section describes a
simpler {\em procedure} that one can attempt in order to obtain a
generating set; at the moment, we do not know if this procedure is
guaranteed to terminate. We conclude the paper with a computational
example for which termination does occur.

\subsection*{Acknowledgments}

We thank a referee for pointing out the recent results \cite{Yamaguchi14}
on the Markov degree of the (generalized) Birkhoff model---which are
stronger than a bound that we had in the original version
of this paper---and for the remark that the integral polytopes
capturing this model are not only normal but even compressed (see
Remark~\ref{re:Compressed}).

\section{Reduction to matching monoids}
\label{sec:Reduction}

In this section we reduce the Main Theorem to a special case to be treated
in the next two sections. To formulate this special case, let $N \in \NN$
and for each $p \in [N]$ let $k_p \in \NN$. First, introduce a set $Y'$
of variables $y'_{p,J}$ where $p \in [N]$ and $J=(j_l)_{l \in [k_p]}
\in \NN^{[k_p]}$ is a $k_p$-tuple of {\em distinct} natural numbers. The
group $\Sym(\NN)$ acts on $Y'$ by $\pi y'_{p,J}=y'_{p,\pi(J)}$ where
$\pi(J)=(\pi(j_l))_{l \in [k_p]}$.  This action has finitely many orbits
and satisfies the condition preceding the Main Theorem.  Second, let
$X$ be a set of variables $x_{p,l,j}$ with $p \in [N], l \in [k_p],
j \in \NN$ and let $\Sym(\NN)$ act on $X$ by its action on the last
index.

\begin{prop} \label{prop:Reduction}
Let $\phi':R[Y'] \to R[X]$ be the $R$-algebra homomorphism sending
$y'_{p,J}$ to $\prod_{l \in [k_p]} x_{p,l,j_l}$. Then the Main Theorem
implies that $\ker \phi'$ is generated by finitely many $\Inc(\NN)$-orbits
of binomials, and that $\im \phi'$ is an $\Inc(\NN)$-Noetherian ring.
Conversely, if these two statements hold for all choices of $N,
k_1,\ldots,k_N \in \NN$, then the Main Theorem holds.
\end{prop}

\begin{proof}
The first statement is immediate---note that the pair $(p,l)$ comprising
the first two indices of the variables $x_{p,l,j}$ takes on
finitely many, namely, $\sum_p k_p$ values.  

For the second statement, consider a monomial map $\phi:R[Y] \to R[Z]$
with $Z=\{z_{i,j} \mid i \in [k], j \in \NN\}$ as in the Main Theorem. Let
$N$ be the number of $\Sym(\NN)$-orbits on $Y$ and let $y_p,\ p \in [N]$
be representatives of the orbits.  Set $k_p:=k_{y_p}$ for $p \in [N],$
so that $\pi y_p$ depends only on the restriction of $\pi \in \Sym(\NN)$
to $[k_p]$. We have thus determined the values of $N$ and the $k_p$,
and we let $Y',X$ be as above.

Let $\psi:R[Y'] \to R[Y]$ be the $R$-algebra homomorphism defined by
sending $y'_{p,J}$ to $\pi y_p$ for any $\pi \in \Sym(\NN)$ satisfying
$\pi(l)=j_l,\ l \in [k_p]$. This homomorphism is $\Sym(\NN)$-equivariant.
The composition $\phi'':=\phi \circ \psi:R[Y'] \to R[Z]$ satisfies the
conditions of the Main Theorem. Since $\psi$ is surjective, it maps any
generating set for $\ker \phi''$ onto a generating set for $\ker \phi$;
moreover, we have $\im \phi''=\im \phi$. Hence the conclusions of the
Main Theorem for $\phi''$ imply those for $\phi$.

Next write $\phi''(y_{p,J}) = \prod_{i \in [k],j \in
\NN}z_{i,j}^{d_{p,i,j}}$. Observe that $d_{p,i,j} = 0$ whenever
$j \not\in J$, using the fact that any permutation that fixes $J$
also fixes $y_{p,J}$, and hence must also fix $\phi''(y_{p,J})$ by
$\Sym(\NN)$-equivariance. Now let $\phi':K[Y'] \to K[X]$ be as above and
define $\rho: R[X] \to R[Z]$ by $\rho(x_{p,l,j}) = \prod_{i \in [k]}
z_{i,j}^{d_{p,i,j}}$. By construction, we have $\rho \circ\phi'
= \phi''$.

Now $\im \phi''$ is a quotient of $\im \phi'$ and $\ker \phi''$ is
generated by $\ker \phi'$ together with pre-images of generators of
$\ker(\rho|_{\im \phi'})$, hence the conclusions of the Main Theorem
for $\phi'$ imply those for $\phi''$, as desired.
\end{proof}

In what follows, we will drop the accents on the $y$-variables and write
$Y$ for the set of variables $y_{p,J}$, $X$ for the set of variables
$x_{p,l,j}$, and $\phi$ for the $R$-algebra homomorphism
\begin{equation} \label{eq:phi}
\phi: R[Y] \to R[X],\ y_{p,J} \mapsto \prod_{l \in [k_p]}
x_{p,l,j_l}. 
\end{equation}
Monomials in the $x_{p,l,j}$ will be denoted $x^A$ where $A \in \prod_{p
\in [N]} \NN^{[k_p] \times \NN}$ is an $[N]$-tuple of finite-by-infinite
matrices $A_p$. Note that $\phi(y_{p,J})$ equals $x^A$ where only the
$p$-th component $A_p$ of $A$ is non-zero and in fact has all row sums
equal to $1,$ all column sums labelled by $J$ equal to $1,$ and all other
column sums equal to $0$. Thus $A_p$ can be thought of as the adjacency
matrix of a matching of the maximal size $k_p$ in the complete bipartite
graph with bipartition $[k_p] \bigsqcup \NN$. Thus the monomials in $\im
\phi$ form the Abelian monoid generated by such matchings (with $p$
varying). We call a monoid like this a \emph{matching monoid}. In the
next section we characterize these monomials among all monomials in the
$x_{p,l,j},$ and find a bound on the relations among the $\phi(y_{p,J})$.

\begin{figure}[h]
  \begin{center}
  \includegraphics[scale=.7]{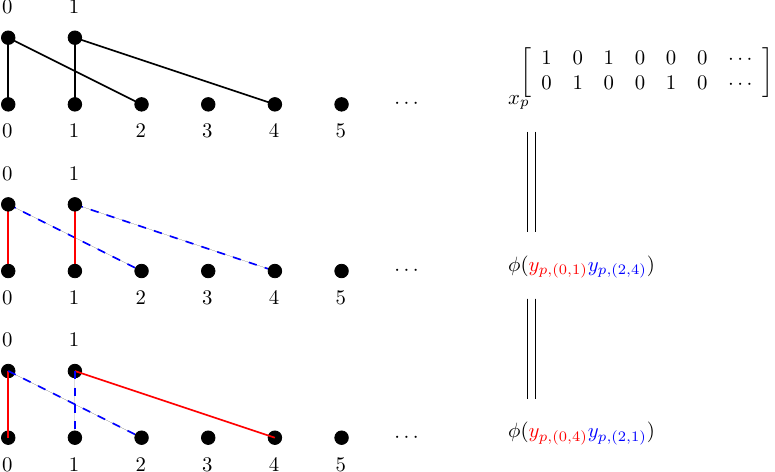}
  \caption{A bipartite graph on $[2] \bigsqcup \NN$ and its
  corresponding monomial $x_p^{A_p}$ (top).  This graph can
  be decomposed into matchings in two different ways (middle
  and bottom).  Each decomposition represents a monomial in the preimage $\phi^{-1}(x_p^{A_p})$.}
  \label{fig:matching}
  \end{center}
\end{figure}

\section{Relations among matchings} \label{sec:Relations}
We retain the setting at the end of the previous section: $Y$ is the
set of variables $y_{p,J}$ with $p$ running through $[N]$ and $J \in
\NN^{[k_p]}$ running through the $[k_p]$-tuples of {\em distinct}
natural numbers; $X$ is the set of variables $x_{p,l,j}$ with $p \in
[N], l \in [k_p], j \in \NN$, and $\phi$ is the map in \eqref{eq:phi}.
In this section we describe both the kernel and the image of $\phi$.
Note that if some $k_p$ is zero, then the corresponding (single)
variable $y_{p,()}$ is mapped by $\phi$ to $1$. The image of $\phi$
does not change if we disregard those $p$, and the kernel changes only
in that we forget about the generators $y_{p,()}-1$. Hence we may and
will assume that all $k_p$ are strictly positive.  The following lemma
gives a complete characterization of the $x^A$ in the image of $\phi$.

\begin{prop}
For an $[N]$-tuple $A \in \prod_{p \in [N]} \NN^{[k_p] \times \NN}$ the
monomial $x^A$ lies in the image of $\phi$ if and only if for all $p \in
[N]$ the matrix $A_p \in \NN^{[k_p] \times \NN}$ has all row sums equal
to a number $d_p \in \NN$ and all column sums less than or equal to $d_p$.
\end{prop}

We call such $A$ {\em good}. Note that $d_p$ is unique since all $k_p$
are strictly positive. We call the vector $(d_p)_p$ the {\em multi-degree}
of $A$ and of $x^A$. 

\begin{re} \label{re:Compressed}
By replacing $\NN$ with $[n]$ for some natural number $n$ greater than
or equal to the maximum of the $k_p$, the proposition boils down to
the statement that for each $p$ the lattice polytope in $\RR^{[k_p]
\times [n]}$ with defining inequalities $\forall_{ij} a_{ij} \geq 0,$
$\forall_i \sum_j a_{ij}=1,$ and $\forall_j \sum_i a_{ij} \leq 1$ is
normal (in the case where $n=k_p$ this is the celebrated {\em Birkhoff
polytope}). This is a not new result; in fact, this polytope satisfies
a stronger property, namely, it is {\em compressed}. This follows,
for instance, from \cite[Theorem 2.4]{Sullivant06} or from the main
theorem of \cite{Ohsugi01}; see also \cite[Section 4.2]{Yamaguchi14}. For
completeness, we include a proof of the proposition using elementary
properties of matchings in bipartite graphs.
\end{re}

\begin{proof}
Let $x_p$ denote the vector of variables $x_{p,l,j}$ for $l \in [k_p]$
and $j \in \NN$. By definition of $\phi,$ the monomial $x^A$ lies in $\im \phi$
if and only if the monomial $x_p^{A_p}$ lies in $\im \phi$ for all $p
\in [N]$. Thus it suffices to prove that $x_p^{A_p}$ lies in $\im \phi$
if and only if all row sums of $A_p$ are equal, say to $d \in \NN,$
and all column sums of $A_p$ are at most $d$. The ``only if'' part is
clear, since every variable $y_{p,J}$ is mapped to a monomial $x_p^B$
where $B \in \NN^{[k_p] \times \NN}$ has all row sums $1$ and all column sums
at most $1$.  For the ``if'' part we proceed by induction on $d$: assume
that the statement holds for $d-1,$ and consider a matrix $A_p$ with row
sums $d$ and column sums $\leq d,$ where $d$ is at least $1$. Clearly, the ``if'' part is true in the case $d = 0$.

Think of $A_p$ as the adjacency matrix of a bipartite graph $\Gamma$ (with multiple
edges) with bipartition $[k_p]\bigsqcup\NN$ (see Figure~\ref{fig:matching}).  With this viewpoint in mind, we will invoke some standard results from combinatorics, and refer to \cite[Chapter 16]{schrijver2003combinatorial}.
The first observation is that $\Gamma$
contains a matching that covers all vertices in $[k_p]$. Indeed, otherwise,
by Hall's marriage theorem, after permuting rows and columns, $A_p$
has the block structure
\[
A_p= \begin{bmatrix}
A_{11} & 0 \\
A_{12} & A_{22}
\end{bmatrix}
\]
with $A_{11} \in \NN^{[l] \times [l-1]}$ for some $l, 1 \leq l \leq
k_p$. But then the entries of $A_{11}$ added row-wise add up to $ld,$
and added column-wise add up to at most $(l-1)d,$ a contradiction.
Hence $\Gamma$ contains a matching that covers all of $[k_p]$.  Next, let
$S \subseteq \NN$ be the set of column indices where $A_p$ has column sum
equal to the upper bound $d$. We claim that $\Gamma$ contains a matching
that covers all of $S$. Indeed, otherwise, again by Hall's theorem, after
permuting rows and columns $A_p$ has the structure
\[
A_p=\begin{bmatrix}
A_{11} & A_{12} \\
0 & A_{22}
\end{bmatrix}
\]
with $A_{11} \in \NN^{[l-1] \times [l]}$ for some $l, 1 \leq l \leq |S|$;
here the first $l$ columns correspond to a subset of the original $S$. Now the
entries of $A_{11}$ added columnwise yield $l d,$ while the entries of
$A_{11}$ added rowwise yield at most $(l-1)d,$ a contradiction.

Finally, we invoke a standard result in matching theory (see
\cite[Theorem 16.8]{schrijver2003combinatorial}), namely that
since $\Gamma$ contains a matching that covers all of $[k_p]$ and a matching
that covers all of $S,$ it also contains a matching that
covers both. Let $B$
be the adjacency matrix of this matching, so that $B$ has all row sums
$1$ and all column sums $\leq 1,$ with equality at least in the columns
labelled by $S$.  Then $A'_p:=A_p-B$ satisfies the induction hypothesis
for $d-1,$ so $x_p^{A_p'} \in \im \phi$. Also, $x_p^B= \phi(y_{p,J})
,$ where $j_a \in \NN$ is the neighbour of $a \in [k_p]$ in the
matching given by $B$. Hence, $x_p^{A_p}=x_p^{A_p'} x_p^B \in \im \phi$ as claimed.
\end{proof}

This concludes the description of the image of $\phi$. For the kernel,
we quote the following result.

\begin{thm}[Theorem 2.1 of \cite{Yamaguchi14}]\label{thm:kernel-of-phi}
The kernel of $\phi$ from \eqref{eq:phi} is generated by binomials in
the $y_{p,J}$ of degree at most $3$.
\end{thm}

Indeed, for each fixed $p$, and replacing $\NN$ by some $[n]$ with
$n \geq k_p$, the monomial map \eqref{eq:phi} captures precisely the
generalization of the Birkhoff model studied in \cite{Yamaguchi14},
where each voter choses $k_p$ among $n$ candidates. Then their Theorem
2.1 yields that the kernel is generated in degrees $2$ and $3$. Since
this holds for each $n \geq k_p$, it also holds for $\NN$ instead of
$[n]$. Moreover, taking the union over all $p$ of sets of generators for
each individual $p$ yields a set of generators for the kernel of $\phi$.
A straightforward consequence of the theorem is the following.

\begin{cor} \label{cor:kernel-of-phi}
The kernel of $\phi$ from \eqref{eq:phi} is generated by finitely many
$\Inc(\NN)$-orbits of binomials.
\end{cor}


\section{Noetherianity of matching monoid rings}
\label{sec:Noetherianity1}

By Corollary~\ref{cor:kernel-of-phi} and Proposition~\ref{prop:Reduction},
Main Theorem follows from the following proposition.

\begin{prop}\label{incnoeth}
The ring $S = R[x^A \mid A \in \prod_{p \in [N]} \NN^{[k_p] \times \NN}\; \mathrm{ good}]$ is $\Inc(\NN)$-Noetherian.
\end{prop}

Let $G \subset \prod_{p\in[N]}\NN^{[k_p] \times \NN}$ be the set of good ($N$-tuples of) matrices, so the monomials of $S$ are precisely $x^A$ for $A \in G$. The monoids $\Sym(\NN)$ and $\Inc(\NN)$
act on $G$ by permuting or shifting columns, so we have $\pi
x^A=x^{\pi A}$, where the $\pi(j)$-th column of the matrix $(\pi A)_p$ equals the
$j$-th column of $A_p$. Let $d_A = (d_{A,p})_p \in \NN^{[N]}$ denote the multi-degree of $A$; recall that this means that all row sums of $A_p$ are equal to $d_{A,p}$. To prove Noetherianity we will
define a partial order $\preceq$ on $G$ and prove that $\preceq$ is a
well-partial-order. Thus we need some basic results from order theory.

A partial order $\preceq$ on a set $P$ is a {\em well-partial-order}
(or {\em wpo}) if for every infinite sequence $p_1,p_2,\ldots$
in $P,$ there is some $i < j$ such that $p_i \preceq p_j$; see
\cite{kruskal1972theory} for alternative characterisations. For
instance, the natural numbers with the usual total order $\leq$ is
a well-partial-order, and so is the componentwise partial order on
the Cartesian product of any finite number of well-partially-ordered
sets. Combining these statements yields Dickson's Lemma~\cite{Dickson13}
that $\NN^k$ is well-partially-ordered. This can be seen as a special
case of Higman's Lemma~\cite{Higman52}, for a beautiful proof of which
we refer to~\cite{nash1963well}.

\begin{lm}[Higman's Lemma]\label{higmans}
 Let $(P, \preceq)$ be a well-partial-order and let $P^* := \bigcup_{l=0}^\infty P^l,$ the set of all finite sequences of elements of $P$.  Define the partial order $\preceq'$ on $P^*$ by $(a_0,\ldots,a_{l-1}) \preceq' (b_0,\ldots,b_{m-1})$ if and only if there exists a strictly increasing function $\rho: [l] \to [m]$ such that $a_j \preceq b_{\rho(j)}$ for all $j \in [l]$.  Then $\preceq'$ is a well-partial-order.
\end{lm}

Our interest in well-partial-orders stems from the following
application. Consider a commutative monoid $\mon$ with an action
of a (typically non-commutative) monoid $\Pi$ by means of monoid
endomorphisms.  We suggestively call the elements of $\mon$
monomials.  Assume that we have a {\em $\Pi$-compatible
monomial order} $\leq$
on $\mon$, i.e., a well-order that satisfies $a < b \Rightarrow
ac < bc$ and $a < b \Rightarrow \pi a < \pi b$ for all $a,b,c \in
\mon$ and $\pi  \in \Pi$. Then it follows that the divisibility
relation $|$ defined by $a|b$ if there exists a $c \in \mon$
with $ac=b$ is a partial order, and also that $a \leq \pi a$ for all $a
\in \mon$. Define a third partial order, the
{\em $\Pi$-divisibility order}, $\preceq$ on $\mon$ by $a
\preceq b$ if there exists a $\pi \in \Pi$ and a $c \in \mon$ such
that $c \pi a = b$. A straightforward computation shows that $\preceq$
is, indeed, a partial order---antisymmetry follows using $a \leq \pi a$.

\begin{prop}\label{wpoNoeth}
If $\preceq$ is a well-partial-order, then for any
Noetherian ring $R$, the $R$-algebra $R[\mon]$ is
$\Pi$-Noetherian.
\end{prop}

\begin{proof}
This statement was proved in \cite{hillar2012finite} for the
case where $R$ is a field.  The more general case can be
proved with the same argument by incorporating work done in
\cite{Aschenbrenner07}.
\end{proof}

Note that the monoid $\{x^A\mid A \in G\}$ that we are considering
here can be given a monomial order which respects the
$\Inc(\NN)$-action. For example, take the lexicographic order, where
the variables $x_{p,i,j}$ are ordered by their indices: $x_{p,i,j} <
x_{p',i',j'}$ if and only if $p < p'$; or $p = p'$ and $j < j'$; or $p =
p',$ $j = j',$ and $i < i'$.

The $\Inc(\NN)$-divisibility order gives a partial order $\preceq$ on
the set $G$ of good ($N$-tuples of) matrices by $A \preceq B$ if and only if there is a
monomial $x^{C} \in S$ and $\pi \in \Inc(\NN)$ such that $x^{C}\pi(x^{A})
= x^{B},$ or equivalently there is $\pi \in \Inc(\NN)$ such that $B -
\pi A \in G$.
Note that $A \preceq B$ not only implies there is some $\pi \in
\Inc(\NN)$ such that all $A_{p,i,j} \leq B_{p,i,\pi(j)},$ but
additionally that all ($N$-tuples of) column sums of $B- \pi A$ are at
most $d_B - d_A \in \NN^{[N]}$.  This prevents us from applying
Higman's Lemma directly to $(G,\preceq)$.  To encode this condition on
column sums, for any $A \in G,$ let $\tilde{A} \in
\prod_{p\in[N]}\NN^{[k_p+1] \times \NN}$ be the $N$-tuple of matrices
such that for all $p\in[N]$, the first $k_p$ rows of $\tilde{A}_p$ are
equal to $A_p$, and the last row of $\tilde{A}_p$ is such that all
column sums equal $d_{A,p}$:
 \[ \tilde{A}_{p,i,j} = \left\{ \begin{array}{ll}  A_{p,i,j} &
\text{for }\; i < k_p, \text{ and}\\
                            d_{A,p} - \sum_{l = 0}^{k_p-1} A_{p,l,j} &
\text{for }\; i = k_p. \end{array}\right. \]
We let $\tilde{G}$ be the set of $N$-tuples of matrices of the form
$\tilde{A}$ with $A \in G$. It is precisely the set of $N$-tuples of
matrices of the form $\tilde{A} \in
\prod_{p\in[N]}\NN^{[k_p+1]\times\NN}$ with the property that there
exists a $d_A \in \NN^{[N]}$ such that for each $p\in[N]$ the first
$k_p$ row sums of $A_p$ are equal to $d_{A,p}$ and all column sums of
$A_p$ are equal to $d_{A,p}$.  Since $A \in G$ has only finitely many
$N$-tuples of non-zero columns, $\tilde{A}$ will have all but finitely
many $N$-tuples of columns equal to
$((0,\ldots,0,d_{A,p})^T)_{p\in[N]}$.  Such $N$-tuples of columns will
be called {\em trivial} (of degree $d_A$).  The $N$-tuple of $j$th
columns of $\tilde{A}$ will be denoted $\tilde{A}_{\cdot\cdot j}$.  We
define the action of $\Inc(\NN)$ on $\tilde{G}$ as $\pi(\tilde{A}) =
\widetilde{\pi(A)}$. Note that for any $j \notin \im(\pi)$, the column
$(\pi \tilde{A})_{\cdot\cdot j}$ is trivial of degree $d_A$, rather than uniformly zero.

\begin{prop}\label{gtilde}
 For $A,B \in G,$ $A \preceq B$ if and only if there is $\pi \in \Inc(\NN)$ such that $\pi\tilde{A} \leq \tilde{B}$ entry-wise.
\end{prop}
\begin{proof}
The condition that $(\pi\tilde{A})_{p,i,j} \leq \tilde{B}_{p,i,j}$ for
all $p \in [N]$, all $i  < k_p$, and all $j \in \NN$ is equivalent to the condition that $B - \pi A$ is non-negative.  Using the fact that
 \[ \tilde{B}_{p,k_p,j} - (\pi\tilde{A})_{p,k_p,j} = (d_{B,p} - d_{A,p}) - \sum_{i = 0}^{k_p-1} (B_p - \pi A_p)_{i,j}, \]
the condition that $\tilde{B}_{p,k_p,j} - (\pi\tilde{A})_{p,k_p,j} \geq 0$ for all $p \in [N]$ and all $j \in \NN$ is equivalent to the condition that every $N$-tuple of column sums of $B - \pi A$ is less than or equal to $d_B - d_A$.  Therefore $\pi\tilde{A} \leq \tilde{B}$ if and only if $B - \pi A \in G$.
\end{proof}

\begin{ex}\label{ex:preceqandsqpreceq}
 Let $A$ and $B$ be the following good matrices in $\NN^{[2] \times \NN}$:
  \[ A = \begin{bmatrix} 3 & 0 & 0 & 0 & 0 & \cdots & \\
                         0 & 1 & 1 & 1 & 0 & \cdots & \end{bmatrix}, \quad
     B = \begin{bmatrix} 3 & 1 & 0 & 0 & 0 & \cdots & \\
                         0 & 2 & 1 & 1 & 0 & \cdots & \end{bmatrix}.
  \]
 Note that $\pi A \leq B$ when $\pi$ is the identity, however $A \not\preceq B$.  Consider
  \[ \tilde{A} = \begin{bmatrix} 3 & 0 & 0 & 0 & 0 & \cdots & \\
                                 0 & 1 & 1 & 1 & 0 & \cdots & \\
                                 0 & 2 & 2 & 2 & 3 & \cdots & \end{bmatrix}, \quad
     \tilde{B} = \begin{bmatrix} 3 & 1 & 0 & 0 & 0 & \cdots & \\
                                 0 & 2 & 1 & 1 & 0 & \cdots & \\
                                 1 & 1 & 3 & 3 & 4 & \cdots & \end{bmatrix},
  \]
 and note that there is no $\pi \in \Inc(\NN)$ such that $\pi \tilde{A} \leq \tilde{B}$.
\end{ex}

We will work with finite truncations of $N$-tuples of matrices in $\tilde{G}$. Let $\matr$
be the set of $N$-tuples of matrices $A \in \bigcup_{\ell=0}^\infty \prod_{p\in[N]}\NN^{[k_p+1] \times
[\ell]}$ such that there exists $d_A \in \NN^{[N]}$ such that for all $p$, all column sums of $A_p$ are equal
to $d_{A,p}$ and the first $k_p$ row sums are {\em at most} $d_{A,p}$; we call $d_A$ the
{\em multi-degree} of $A$.  Note that the condition on row sums is relaxed,
which will allow us to freely remove columns from matrices while still
remaining in the set $\matr$. For $A \in \matr$ the number of columns of $A$
is called the {\em length} of $A$ and denoted $\ell_A$.
We give $\matr$ the
partial order $\preceq$ defined as follows.  For $A,B \in \matr,$ $A \preceq
B$ if and only if there is a strictly increasing map $\rho:[\ell_A]
\to [\ell_B]$ such that $\rho A \leq B$.  Just as in $\tilde{G},$ here
$\rho A$ is defined by $(\rho A)_{\cdot\cdot j} = A_{\cdot\cdot\rho^{-1}(j)}$ for $j \in
\im(\rho),$ and $(\rho A)_{\cdot\cdot j}$ trivial (of degree $d_A$) for $j \in [\ell_B] \setminus
\im(\rho)$.  For an $N$-tuple of matrices $A$ and a set $J \subset \NN,$ let $A_{\cdot\cdot J}$
denote the $N$-tuple of matrices obtained from $A$ by taking only the columns $A_{\cdot\cdot j}$
with $j \in J$.

Some care must be taken in the definition of $\matr$ since we allow matrices with no columns.  In all other cases, the degree of $A \in \matr$ is uniquely determined by its entries.  However for the length 0 case the degree is arbitrary, so we will consider $\matr$ as having a distinct length 0 element $Z^d$ with degree $d$ for each $d \in \NN^{[N]},$ and we define $Z^d \preceq A$ if and only if $d \leq d_A$.  Additionally, define $A_{\cdot\cdot\emptyset} = Z^{d_A}$.

\begin{de}
 For $A \in \matr,$ the $N$-tuple of $j$th columns of $A$ is {\em bad} if for some $p \in [N]$, we have $A_{p,k_p,j} < d_{A,p}/2$.  If $A_{p,k_p,j} < d_{A,p}/2$, we will call $j$ a {\em bad index} of $A$ (with respect to $p$). Let $\matr_t$ denote the set of $N$-tuples of matrices in $\matr$ with exactly $t$ bad indices.
\end{de}

We will use induction on $t$ to show that $(\matr_t,\preceq)$ is well-partially ordered for all $t \in \NN$.  This will in turn be used to prove that $(\matr, \preceq)$ and then $(\tilde{G},\preceq)$ are well-partially ordered.  First we prove the base case:

\begin{prop}\label{basecase}
 $(\matr_0,\preceq)$ is well-partially ordered.
\end{prop}
\begin{proof}
 Let $A^{(1)},A^{(2)},\ldots$ be any infinite sequence in $\matr_0$. We will show that there are $r < s$ such that $A^{(r)} \preceq A^{(s)}$.

 Fix $p \in [N]$. There are now two possibilities: either the degrees of the elements of the sequence $A_p^{(1)},A_p^{(2)},\ldots$ are bounded by some $d_p \in \NN$, or they are not. In the former case, it follows that the number of non-trivial columns in any $A_p^{(r)}$ is bounded by $d_pk_p$. Then there is a subsequence $B_p^{(1)},B_p^{(2)},\ldots$ of $A_p^{(1)},A_p^{(2)},\ldots$ such that every element has the same degree and same number of non-trivial columns. In the latter case, $A_p^{(1)},A_p^{(2)},\ldots$ has a subsequence with strictly increasing degree and moreover a subsequence $B_p^{(1)},B_p^{(2)},\ldots$ with the property that $d_{B^{(s+1)},p} \geq 2d_{B^{(s)},p}$ for all $s \in \NN$.

In either case, without loss of generality, we replace $A^{(1)},A^{(2)},\ldots$ by $B^{(1)},B^{(2)},\ldots$. We repeat this procedure for all $p\in [N]$, and we find that $A^{(1)},A^{(2)},\ldots$ contains a subsequence $B^{(1)},B^{(2)},\ldots$ such that for all $p \in [N]$, one of the following two statements holds.

\begin{description}
\item[1] Both $d_{B^{(t)},p}$ and the number of non-trivial columns in $B_p$ are constant.

\item[2] We have $d_{B^{(t+1)},p} \geq 2d_{B^{(t)},p}$ for all $t$.
\end{description}

It now suffices to show that there are $r < s$ such that $B^{(r)} \preceq B^{(s)}$ for all $r < s$. Define the partial order $\sqsubseteq$ on $\matr_0$ by $A \sqsubseteq B$ if and only if there exists strictly increasing $\rho:[\ell_A] \to [\ell_B]$ such that $A_{\cdot\cdot j} \leq B_{\cdot\cdot\rho(j)}$ for all $j \in [\ell_A]$.  By Higman's Lemma (Lemma~\ref{higmans}), $\sqsubseteq$ is a wpo. This means that there exist $r < s$ such that $B^{(r)} \sqsubseteq B^{(s)}$. Fix such a pair $r < s$. We will show that $B^{(r)} \preceq B^{(s)}$.

Let $\rho: [\ell_{B^{(r)}}] \to [\ell_{B^{(s)}}]$ be a strictly increasing map that witnesses $B^{(r)} \sqsubseteq B^{(s)}$. We claim that it also witnesses $B^{(r)} \preceq B^{(s)}$. For this, we have to show that $\rho B^{(r)} \leq B^{(s)}$. By the properties of $\sqsubseteq$, we already have $(\rho B^{(r)})_{\cdot\cdot \rho(j)} \leq B^{(s)}_{\cdot\cdot\rho(j)}$, which is to say that it suffices to show that for all $j\notin \im(\rho)$, we have $d_{B^{(r)}} \leq (B^{(s)}_{p,k_p,j})_{p\in [N]}$.

Let $p \in [N]$. Suppose we are in the case that both $d_{B^{(t)},p}$ and the number of non-trivial columns in $B_p$ are constant. Since $\rho$ must map non-trivial columns of $B_p^{(r)}$ to non-trivial columns of $B_p^{(s)}$, we conclude that if $j\notin \im(\rho)$, then the $j$-th column of $B_p^{(s)}$ is trivial, and hence $(B^{(s)}_{p,k_p,j}) = d_{B^{(s)},p}$. But the latter equals $d_{B^{(r)},p}$, so certainly $d_{B^{(r)},p} \leq (B^{(s)}_{p,k_p,j})$.

Alternatively, suppose we have $d_{B^{(t+1)},p} \geq 2d_{B^{(t)},p}$ for all $t$. Since $B_p^{(s)}$ has no bad columns, we have
\[ B^{(s)}_{p,k_p,j} \geq \frac{1}{2}d_{B^{(s)},p} \geq d_{B^{(r)},p}. \]
This is exactly what we wanted to show.

So in both cases, we find that $d_{B^{(r)},p} \leq B^{(s)}_{p,k_p,j}$ for all $j \notin \im(\rho)$. This is true for all $p$, so we have $d_{B^{(r)}} \leq (B^{(s)}_{p,k_p,j})_{p\in [N]}$. We conclude that $B^{(r)} \preceq B^{(s)}$, as we wanted to show.
\end{proof}

\begin{prop}\label{indstep}
 $(\matr_t,\preceq)$ is well-partially ordered for all $t \in \NN$.
\end{prop}
\begin{proof}
 The base case, $t = 0,$ is given by Proposition~\ref{basecase}.  For $t > 0,$ assume by induction that $(\matr_{t-1},\preceq)$ is well-partially ordered.
 For any $A \in \matr_t,$ let $j_A$ be the largest bad index of $A$.
Then $A$ can be decomposed into three parts: the $N$-tuple of matrices
of all $N$-tuples of columns before $j_A$, $A_{\cdot\cdot j_A}$ itself, and the $N$-tuple of matrices of all $N$-tuples of columns after $j_A$.  This decomposition is represented by the map
  \[ \delta:\matr_t \to \matr_{t-1} \times \prod_{p\in[N]}\NN^{[k_p+1]} \times \matr_0 \]
  \[ A \mapsto (A_{\cdot\cdot\{0,\ldots,j_A-1\}}, A_{\cdot\cdot j_A}, A_{\cdot\cdot\{j_A+1,\ldots,\ell_A-1\}}). \]
 Let the partial order $\sqsubseteq$ on $\matr_{t-1} \times \prod_{p\in[N]}\NN^{[k_p+1]} \times \matr_0$ be the product order of the wpos $(\matr_{t-1},\preceq),$ $(\NN^{[k+1]},\leq)$ and $(\matr_0,\preceq)$.  Note that the product order of any finite number of wpos is also a wpo.  Suppose for some $A,B \in \matr_t$ that $\delta(A) \sqsubseteq \delta(B)$.  This implies that $A_{\cdot\cdot j_A} \leq B_{\cdot\cdot j_B}$ and that there exist strictly increasing maps $\rho$ and $\sigma$ such that $\rho (A_{\cdot\cdot[j_A]}) \leq B_{\cdot\cdot[j_B]}$ and $\sigma (A_{\cdot\cdot\{j_A+1,\ldots,\ell_A-1\}}) \leq B_{\cdot\cdot\{j_B+1,\ldots,\ell_B-1\}}$.  We combine these into a single strictly increasing map $\tau:[\ell_A] \to [\ell_B]$ defined by
  \[ \tau(j) = \left\{ \begin{array}{ll}  \rho(j) & \text{for }\; 0 \leq j < j_A \\
                                              j_B & \text{for }\; j = j_A \\
                            \sigma(j-j_A-1)+j_B+1 & \text{for }\; j_A < j < \ell_A \end{array},\right. \]
 illustrated in Figure \ref{fig:po}.  Then $\tau A \leq B$ so $A \preceq B$.    Since $\sqsubseteq$ is a wpo, $(\matr_t, \preceq)$ is also a wpo.
\end{proof}

\begin{figure}[h]
  \centering
  \includegraphics[width=.8\columnwidth]{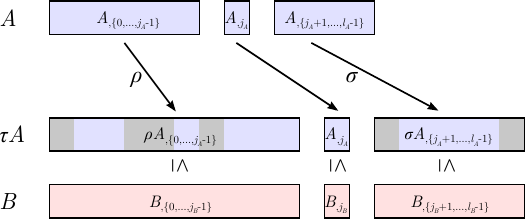}
  \caption{$\delta(A) \sqsubseteq \delta(B)$ implies $A \preceq B$.}
  \label{fig:po}
\end{figure}

\begin{prop}\label{mprimewpo}
 $(\matr,\preceq)$ is well-partially ordered.
\end{prop}
\begin{proof}
 For any $A \in \matr,$ if $j$ is a bad index of $A,$ then for some $p\in [N]$, we have $d_{A,p}/2 > \sum_{i \in [k_p]} A_{p,i,j}$.  Letting $J_p \subset \NN$ be the set of bad indices of $A$ with respect to $p$ and let $J \subset \NN$ be the union of the $J_p$. Then
  \[ |J_p|\frac{d_{A,p}}{2} < \sum_{j \in J_p}\sum_{i \in [k_p]} A_{p,i,j} \leq \sum_{i \in [k_p]} \sum_{j \in \NN} A_{p,i,j} \leq k_pd_A \]
 with the last inequality due to the row sum condition on $A_p$.  Therefore $|J_p| \leq 2k_p-1$, and hence $|J| \leq 2\sum_{p\in [N]}k_p - N$.

 Let $A^{(1)},A^{(2)},\ldots$ be any infinite sequence in $\matr$.  Since the numbers of bad $N$-tuples of columns of elements of $\matr$ are bounded by $2\sum_{p\in [N]}k_p - N$ there exists a subsequence which is contained in $\matr_t$ for some $0 \leq t \leq 2\sum_{p\in [N]}k_p - N$.  By Proposition~\ref{indstep} there is $r < s$ with $A^{(r)} \preceq A^{(s)}$.
\end{proof}

\begin{prop}\label{mwpo}
 $(G,\preceq)$ is well-partially ordered.
\end{prop}
\begin{proof}
 Let $A^{(1)}, A^{(2)}, \ldots$ be any infinite sequence in $G$.  Each $A^{(r)}$ has some $j_r > 0$ such that all $N$-tuples of columns $A^{(r)}_{\cdot\cdot m}$ are zero for $m \geq j_r$.  Consider the sequence $\tilde{A}^{(1)}_{\cdot\cdot[j_1]}, \tilde{A}^{(2)}_{\cdot\cdot[j_2]},\ldots$ in $\matr$ obtained by truncating each $\tilde{A}^{(r)}$ to the first $j_r$ $N$-tuples of columns.  By Proposition~\ref{mprimewpo} there is some $r < s$ and $\rho:[j_r] \to [j_s]$ such that $\rho \tilde{A}^{(r)}_{\cdot\cdot[j_r]} \leq \tilde{A}^{(s)}_{\cdot\cdot[j_s]}$.  Note that this implies $d_{A^{(r)}} \leq d_{A^{(s)}}$.  Extend $\rho$ to some $\pi \in \Inc(\NN)$ so then
  \[ (\pi \tilde{A}^{(r)})_{\cdot\cdot[j_s]} = \rho (\tilde{A}^{(r)}_{\cdot\cdot[j_r]}) \leq \tilde{A}^{(s)}_{\cdot\cdot[j_s]}. \]
  The remaining $N$-tuples of columns of $\pi \tilde{A}^{(r)}$ and $\tilde{A}^{(s)}$ are trivial, so $\pi \tilde{A}^{(r)} \leq \tilde{A}^{(s)}$ follows from the fact that $d_{A^{(r)}} \leq d_{A^{(s)}}$.  Therefore $A^{(r)} \preceq A^{(s)}$ by Proposition~\ref{gtilde}.
\end{proof}

Now we can apply Proposition~\ref{wpoNoeth} to the monoid $\{x^A\mid A \in G\}$ which proves that the ring $R[x^A\mid A \in G]$ is $\Inc(\NN)$-Noetherian.  This concludes the proof of Proposition~\ref{incnoeth}.

\section{Buchberger's algorithm for matching monoid algebras} \label{sec:Computational}

Assume the general setting of Proposition~\ref{wpoNoeth}: $\monoid$ is
a monoid with $\Pi$-action and $\Pi$-compatible monomial order $\leq$.
For a polynomial $f$ and an ideal $I$ in $\mring{K}$, we can define
$\LM(f),$ $\LC(f),$ $\ini(I),$ division with remainder, and the concept of
equivariant Gr\"obner basis from \cite{Brouwer09e}; all relative to the
monomial order $\leq$. 
We now derive a version of Buchberger's algorithm
for computing such a Gr\"obner basis, under an additional assumption.
For $a,b\in \monoid$ we define the set of {\em least common
multiples}
\[
\lcm(a,b) = \{l \in \monoid: a|l, b|l \text{ and } (a|l',b|l',l'|l \Rightarrow l'=l)\}.
\]
We require the following
variant
of
conditions EGB3 and EGB4 from \cite{Brouwer09e}:
\begin{quote}
EGB34. For all $f,g \in \mring{K}$, the set of triples in
$\monoid\times \Pi f \times \Pi g$ defined by
\[
T_{f,g} = \{ (l',f',g') \mid f' \in \Pi f,\; g' \in \Pi g,\; l'\in\lcm(\LM(f'),\LM(g')) \},
\]
is a union of a finite number of $\Pi$-orbits:
\[
T_{f,g} = \bigcup_{i} \Pi (l_i,f_i,g_i),\quad i\in[r].
\]
\end{quote}
In particular, EGB34 implies that for all $a,b\in \monoid$ and $\pi \in \Pi$,
we have $\pi \lcm(a,b) \subseteq \lcm(\pi a,\pi b)$. (This is what
condition EGB3 of \cite{Brouwer09e} looks like when least common
multiples are not unique.)


If EGB34 is fulfilled, then there is a unique {\em inclusion-minimal} finite
set of orbit generators as above, which we denote
\[
O_{f,g} = \{(l_i,f_i,g_i) \mid  i\in[r]\}.
\]
Indeed, suppose that $O$ and $O'$ are both inclusion-minimal sets of
orbit generators for $T_{f,g}$.  For any triple $t \in O$, there are
$t' \in O', \pi \in \Pi$ such that $\pi t' = t$, and similarly $t''
\in O,\tau \in \Pi$ such that $\tau t'' = t'$. Now $t=\pi\tau t''$
and since $O$ is minimal, $t=t''$. But since $\Pi$ is compatible with
a monomial order, $\pi \tau t''=t''$ implies that also the
intermediate expression $t'=\tau t''$ equals $t''$. Hence $O \subseteq
O'$ and equality holds by minimality of $O'$. 

\begin{de}\label{de:s-polys}
For monic $f,g\in \mring{K}$ define the {\em set of S-polynomials} to be
\[
S_{f,g} = \{a f' - b g' \mid (l',f',g') \in O_{f,g};\,  a,b\in\monoid; \text{ and } a \LM(f') = b \LM(g') = l'\}.
\]
Furthermore, define {\em $\Pi$-reduction} of a polynomial $f$ with
respect to a set $G \subseteq \mring{K}$ as follows: while there
exist $g \in G$ and $\pi \in \Pi$ with $\pi\LM(g)|\LM(f)$,
replace $f$ by 
\[ f' := f - \frac{\LC(f)\LM(f)}{\LC(g)\pi\LM(g)}\pi g; \]
and when no such $g$ and $\pi$ exist, return the {\em remainder} $f'$.
\end{de}

One can generalize Gr\"obner theory to our equivariant setting for a monoid algebra satisfying EGB34. In particular, Buchberger's criterion holds, and the following procedure produces an equivariant Gr\"obner basis {\em if} it terminates. 
\begin{alg}\label{alg:Buchberger}
$G = \Buchberger(F)$
\begin{algorithmic}[1]
\REQUIRE $F$ is a finite set of monic elements in $\mring{K},$
the algebra of a monoid $\monoid$ equipped with a $\Pi$-action, satisfying the assumptions above and the condition EGB34.
\ENSURE $G$ is an equivariant Gr\"obner basis of $\ideal{F}$.

\smallskip \hrule \smallskip

\STATE $G\gets F$
\STATE $S\gets \bigcup_{f,g\in G} S_{f,g}$\COMMENT{in particular, compute $O_{f,g}$ needed in Definition~\ref{de:s-polys}}
\WHILE{$S\neq\emptyset$}
	\STATE pick $f \in S$
	\STATE $S\gets S\setminus\{f\}$
	\STATE $h \gets$ the $\Pi$-reduction of $f$ with respect to $G$
  	\IF{$h \neq 0$}
		\STATE $G\gets G\cup \{h\}$
		\STATE $S\gets S\cup \left(\bigcup_{g\in G}S_{g,h}\right)$
	\ENDIF
\ENDWHILE
\smallskip \hrule \smallskip
\end{algorithmic}
\end{alg}

This algorithm has been implemented for the particular case where
$\mring{K}$ is a polynomial ring and $\Pi = \IncNN$ (i.e. the
algorithm described in \cite{Brouwer09e}) in the package {\em
EquivariantGB}~\cite{EquivariantGB} for the computer algebra system
{\em Macaulay2} \cite{grayson2002macaulay}. When the
algebra $\mring{K}$ is $\Pi$-Noetherian, termination of
Algorithm~\ref{alg:Buchberger} is guaranteed, but in general we cannot
make this claim. 

\medskip
We now turn our attention to the task of computing a finite
$\Inc(\NN)$-generating set of binomials of a general toric map as in
the Main Theorem. By the proof of Proposition~\ref{prop:Reduction}
we may assume that $Y$ is as in \eqref{eq:phi}, i.e., it consists of
variables $y_{p,J}$ where $p$ runs through $[N]$ and $J$ runs through
all $k_p$-tuples of distinct natural numbers. Section~\ref{sec:Reduction}
then leads to the following analysis of this task.

\begin{prb}\label{prb:main}
Fix the names of algebras and maps in the following diagram:
\[ R[Y] \xrightarrow{\phi} R[X]\xrightarrow{\psi} R[Z]. \]
Here $\phi$ is the map defined by (\ref{eq:phi}), whose image
is the $R$-algebra spanned by the matching monoid, and $\psi$ is any
$\Sym(\NN)$-equivariant monomial map from $R[X]$ to $R[z_{ij} \mid i \in
[k], j \in \NN]$. For $\ker(\psi\circ\phi)$, how does one compute 
\begin{itemize}
  \item[(a)] a finite set of generators up to $\Inc(\NN)$-symmetry?
  \item[(b)] a finite $\Inc(\NN)$-Gr\"obner basis with respect to a given
  $\Inc(\NN)$-compatible monomial order on $K[Y]$?
\end{itemize}
\end{prb}


The algorithm we are about to construct solves Problem~\ref{prb:main}(a);
indeed, we do not know whether a finite $\Inc(\NN)$-Gr\"obner basis as
in part (b) exists! Our algorithm relies on the fact that we may replace
$R[X]$ above by the matching monoid algebra $\im \phi=R[x^A : A
\text{ good}]$, so as to get the sequence
\begin{equation}\label{eq:YXZ} 
R[Y] \xrightarrow{\phi} R[x^A \mid A \text{ good}]
\xrightarrow{\psi} R[Z].
\end{equation}
Most of our computations will take place in the ring $R[x^A \mid A \text{
good}][Z]$, which is itself a matching monoid with $N$ replaced by $N+k$
and $k_p=1$ for $p \in [N+k]\setminus [N]$. This monoid is Gr\"obner
friendly by the following proposition.

\begin{prop}\label{prop:lcm}
Let $\monoid$ be a submonoid of $\NN^{[k] \times \NN}$ that is generated by the
$\SymNN$-orbits of a finite number of matrices.  For $\Pi = \IncNN$, the monoid algebra $\mring{K}$ satisfies EGB34.
\end{prop}
\begin{proof}
Any such $\mring{K}$ is the image of some map $\phi$ as in the Main Theorem (with $R = K$), and so is $\IncNN$-Noetherian.  Similarly $K[\monoid^3] = K[M]^{\otimes 3}$ is $\IncNN$-Noetherian.  For any $a,b \in \monoid$, the monomial ideal $\ideal{T_{a,b}} \subseteq K[\monoid^3]$ is $\IncNN$-stable.  Let $L \subseteq T_{a,b}$ be a minimal finite $\IncNN$-generating set of $\ideal{T_{a,b}}$.

For any $(l, \pi a, \sigma b) \in T_{a,b}$, there is some
$(m,a',b') \in L$ and $\tau \in \IncNN$ such that
$\tau(m,a',b') | (l, \pi a, \sigma b)$.  It is clear that
$\tau a' = \pi a$ and $\tau b' = \sigma b$. Since $a'$ and
$b'$ divide $m$, $\pi a$ and $\sigma b$ must divide $\tau
m$, and in turn $\tau m$ divides $l$.  But $l \in \lcm(\pi
a, \sigma b)$ by assumption, so $l = \tau m$.  Therefore
$(l, \pi a, \sigma b) = \tau(m,a',b')$.  This shows that
$T_{a,b}$ is the union of the $\IncNN$-orbits of the
elements of $L$, and then $L = O_{a,b}$.

To establish the same fact for a general pair $f,g\in \mring{K}$ we first determine $O_{a,b}$ where $a = \LM(f)$ and $b = \LM(g)$.  For any $(l,\pi f, \sigma g) \in T_{f,g}$, the triple $(l,\pi a, \sigma b) \in T_{a,b}$ is in the orbit of some $(m,a',b') \in O_{a,b}$.  This implies $a' = \tau a$ for some $\tau \in \IncNN$, but $\tau$ is not unique.  Define
 \[ \Lambda_{a,a'} := \{\tau \in \IncNN \mid a' = \tau a; \text{ and } n \in \im \tau \text{ for all }n > \ell_{a'} \}. \]
Here $\ell_{a'}$ denotes the {\em length} of $a'$ as in Section \ref{sec:Noetherianity1}, the maximum index value among all non-zero columns of $a'$.  Note that $\Lambda_{a,a'}$ is a finite set.

Since $\pi a$ is in the orbit of $a'$, $\pi$ factors through some $\tau \in \Lambda_{a,a'}$.  So $(l,\pi f, \sigma g) = (\gamma m, \alpha \tau_1 f, \beta \tau_2 g)$ for some $\gamma, \alpha, \beta \in \IncNN$, $\tau_1 \in \Lambda_{a,a'}$ and $\tau_2 \in \Lambda_{b,b'}$.  Therefore
 \[ T_{f,g} \subseteq \bigcup_{(m,f',g') \in U_{f,g}} \Pi m \times \Pi f' \times \Pi g' \]
where
 \[ U_{f,g} = \bigcup_{(m,a',b') \in T_{a,b}} \{(m,\tau_1 f, \tau_2 g) \mid \tau_1 \in \Lambda_{a,a'}, \tau_2 \in \Lambda_{b,b'}\}. \]
For each $(m,f',g')$, the set $\Pi m \times \Pi f' \times \Pi g'$ is the union of a finite number of $\IncNN$-orbits.  To prove this one can follow closely the proof of \cite{Brouwer09e} Lemma 3.4.  From the finite set of generators we select only those $(\gamma m, \alpha f', \beta g')$ with $\gamma m \in \lcm(\alpha f', \beta g')$, and call this set $O_{(m,f',g')}$.  Then
 \[ O_{f,g} = \bigcup_{(m,f',g') \in U_{f,g}} O_{(m,f',g')} \]
is as desired.
\end{proof}

\begin{alg}\label{alg:toricIdeal}
$T = \toricIdeal(\phi)$
\begin{algorithmic}[1]
\REQUIRE $\phi : R[Y] \to R[Z]$ is a monomial map as in the Main Theorem.
\ENSURE $T$ is a finite set of generators of $\ker \phi$ as
$\Inc(\NN)$-stable ideal.

\smallskip \hrule \smallskip

\STATE Replace $Y$ by the set of variables $\{y_{p,J}\}_{p,J}$ as in
the proof of Proposition~\ref{prop:Reduction}.

\STATE Decompose $\phi$ with the composition of two maps
$\phi$ and $\psi$ as in diagram (\ref{eq:YXZ}).

\STATE Consider the ideal $I_\psi \subset R[x^A \mid A \text{ good}][Z]$
generated by the finite set $F$ of binomials $\psi(x^A)-x^A,$ where $A \in
\prod_{p \in [N]} \NN^{[k_p] \times \NN}$ is good of multi-degree \[
d \in \{(1,0,\ldots,0),\ldots,(0,0,0,\ldots,1)\} \subseteq \NN^{[N]}
\] and $\preceq$-minimal; the $\IncNN$-orbits of such monomials $x^A$
generate $R[x^A \mid A \text{ good}]$.


\STATE Run Algorithm~\ref{alg:Buchberger} for the input $F$ with respect
to a monomial order that eleminates the variables $Z$. Since $R[x^A
\mid A \text{ good}][Z]$ is the monoid algebra of a monoid where
$\Inc(\NN)$-divisibility is a w.p.o., the algorithm
terminates. Standard elimination theory
implies that $G' = G\cap R[x^A \mid A \text{ good}]$ generates 
\[I_\psi \cap R[x^A \mid A \text{ good}] = \ker \psi \cap \im \phi.\]

\STATE Let $T$ consist of preimages of elements in $G'$ (one per element)
and a finite number of binomials whose orbits generating $\ker \phi$
(see Corollary~\ref{cor:kernel-of-phi}).

\smallskip \hrule \smallskip
\end{algorithmic}
\end{alg}

\begin{re}We can execute Algorithm~\ref{alg:toricIdeal} for any coefficient ring $R$ (not necessarily a field), since all polynomials that appear in computation are binomials with coefficients $\pm 1$.\end{re}

In the following two remarks we comment on two major subroutines not spelled out in the sketch of the algorithm above.
\begin{re}\label{re:s-polynomials}
Unlike in the usual Buchberger algorithm, the task of computing
S-polynomials in Algorithm~\ref{alg:Buchberger} is far from being
trivial. To accomplish that, one needs to compute the set $O_{f,g}$, which can be done following the lines of the proof of Proposition~\ref{prop:lcm}. While this procedure is {\em effective}, by no means it is {\em efficient}.
\end{re}

\begin{re} \label{re:preimage}
In the last step of Algorithm~\ref{alg:toricIdeal} a preimage $\phi^{-1}(g)$ of an element $g\in G$ can be computed by reducing the problem to one of computing maximal matchings of bipartite graphs, a well studied problem in combinatorics.  Any monomial $x^A \in \im \phi$ can be considered as a collection of $N$ bipartite graphs with adjacency matrices $A_0,\ldots,A_{N-1}$ as in Section~\ref{sec:Relations}, where each $A_p$ has bipartition $[k_p] \sqcup \NN$.  Fixing $A_p,$ let $S \subset \NN$ be the set of vertices in the second partition with degree $d_{A_p}$ (i.e. the indices of the columns of $A_p$ with column sum equal to $d_{A_p}$).  A matching $B$ covering $[k_p]$ and $S$ can be computed using the Hungarian method or other algorithms for computing weighted bipartite matchings (see \cite{schrijver2003combinatorial} Chapter 17 for more details).  The matching $B$ directly corresponds to a variable $y_{p,J} \in Y$ with $\phi(y_{p,J}) = x^B$.  Since $B$ covers $S,$ it follows that $A_p- B$ is a good matrix.  Therefore $x^{A_p}/\phi(y_{p,J})$ is also in $\im \phi$ and can be decomposed further by repeating the process.
\end{re}

Algorithm~\ref{alg:toricIdeal} has an important theoretical
consequence, a solution to Problem~\ref{prb:main}(a): a finite
$\Inc(\NN)$-generating set of the toric ideals in the Main Theorem is
computable. However, in view of Remark~\ref{re:s-polynomials} and a
more elementary procedure (albeit with no termination guarantee)
given in the following section that solves a
harder Problem~\ref{prb:main}(b) for a small example, we postpone a
practical implementation of Algorithm~\ref{alg:toricIdeal}. 

\section{An example, and a more na\"ive implementation}
\label{sec:Example}

A more elementary approach to
Problem~\ref{prb:main}---indeed, to the hardest variant---is,
for a given order on $[Y,Z]$, to directly apply the algorithm of
\cite{Brouwer09e} to the graph of the entire map $\psi \circ
\phi$,
rather than computing generators for the kernels of $\psi$ and $\phi$
separately as in Algorithm~\ref{alg:toricIdeal}. The advantages of this
approach are that it is simpler to implement, and that it produces not
just a generating set, but an $\Inc(\NN)$-equivariant Gr\"obner basis.
The disadvantage is that {\em we do not know} whether the procedure is
guaranteed to terminate. We now set up a version of the usual equivariant
Buchberger algorithm that is particularly easy to implement, and conclude
with one nontrivial computational example.

\bigskip

For convenience let $\omega = \psi \circ \phi$.  Let $I_\omega \subset
R[Y,Z]$ be the ideal corresponding to the graph of $\omega,$ so
$I_{\omega}$ is generated by the binomials of the form $y - \omega(y)$
for each variable $y \in Y$.  Choosing a representative $y_p =
y_{p,(0,\ldots,k_p -1)}$ of each $\Sym(\NN)$-orbit in $Y,$ the ideal is
$\Inc(\NN)$-generated by the finite set
\[
 F := \{ \sigma y_p - \omega(\sigma y_p) \mid p \in [N],\: \sigma \in \Sym([k_p])\}.
\]
Choose an $\Inc(\NN)$-compatible monomial order $\leq$ on $R[Y,Z]$
that eliminates $Z$. Then apply to $F$ the equivariant Gr\"obner
basis algorithm from \cite{Brouwer09e} (which is essentially
Algorithm~\ref{alg:Buchberger}). Note that since we are working in a
polynomial ring $R[Y,Z],$ rather than a more complicated monoid ring
$R[X \mid X \text{ good}][Z],$ every pair of monomials has only one $\lcm$, which is is
straightforward to compute.  If the procedure terminates with output
$G,$ then $G \cap R[Y]$ is an $\Inc(\NN)$-equivariant Gr\"obner basis
of $I_{\omega} \cap R[Y] = \ker \omega$.

This procedure can be adapted to make use of existing, fast
implementations of traditional Gr\"obner basis algorithms.  For each $n
\in \NN$ truncate to the first $n$ index values by defining
$Y_n := \{ y_{p,J} \mid J \in [n]^{k_p}\},$
$Z_n := \{z_{i,j} \in Z \mid j \in [n]\},$ and\\
$F_n := \{ y - \omega(y) \mid y \in Y_n\}.$
Let $I_n$ be the ideal in $R[Y_n,Z_n]$ generated by $F_n$.  Note that
each $I_n$ is $\Sym([n])$-stable and that $\bigcup_{n \in \NN} I_n =
I_{\omega}$.  Let $\Inc(m,n)$ denote the set of all strictly increasing
maps $[m] \to [n]$, and equip $K[Y_n, Z_n]$ with the
restriction of the $\Inc(\NN)$-monomial order $\leq$.

\begin{alg}\label{alg:naive}
$G = \truncatedBuchberger(\omega)$
\begin{algorithmic}
\REQUIRE $\phi : R[Y] \to R[Z]$ is a monomial map in the Main Theorem.
\ENSURE $G$ is an $\Inc(\NN)$-equivariant Gr\"obner basis of $\ker \phi$.

\smallskip \hrule \smallskip

\STATE $n \gets \max_{p \in [N]} k_p$

\WHILE{true}
	\STATE $F_n \gets \{ y - \omega(y) \mid y \in Y_n\}$
	\STATE $G_n \gets \GrobnerBasis(F_n)$
	\STATE $m \gets \lfloor (n+1)/2 \rfloor$
  	\IF{$m \geq \max_{p \in [N]} k_p$ and $G_n = \Inc(m,n)G_m$}
		\STATE $G \gets G_m \cap R[Y]$
		\STATE return $G$
	\ENDIF
	\STATE $n \gets n+1$
\ENDWHILE

\smallskip \hrule \smallskip
\end{algorithmic}
\end{alg}

Here $\GrobnerBasis$ denotes any algorithm to compute a
traditional Gr\"obner basis.  If
$\truncatedBuchberger(\omega)$ terminates, this implies that
there is some $m \geq \max_{p \in [N]} k_p$ such that
$\Inc(m,n)G_m$ satisfies Buchberger's criterion for some $n
\geq 2m -1$.  Then $G_m$ satisfies the equivariant
Buchberger criterion, so $G_m$ is an equivariant Gr\"obner basis.  Because we require that $m \geq \max_{p \in [N]} k_p,$ the set $G_m$ generates $I_{\omega}$ up to $\Inc(\NN)$-action.  Finally $G = G_m \cap R[Y]$ is an equivariant Gr\"obner basis for $\ker \omega$.

\begin{ex}\label{ex:Sym21}
Set $Y := \{y_{j_0,j_1} \mid j_0,j_1 \in \NN,\; j_0\neq j_1\}$ and $Z :=
\{z_i \mid i \in \NN\},$ each consisting of a single $\Sym(\NN)$-orbit,
and define the monomial map $\omega:R[Y] \to R[Z]$ by
  \[ \omega: y_{j_0,j_1} \mapsto z_{j_0}^2 z_{j_1}. \]
Whether $\ker \omega$ is finitely generated was posed as an open question
in \cite{Hillar13} (Remark 1.6).  This is answered in the affirmative
by Theorem~\ref{thm:main}, but by applying Algorithm~\ref{alg:naive}
we have also explicitly computed an $\Inc(\NN)$-equivariant Gr\"obner
basis.  The Gr\"obner basis computations were carried out using the
software package {\em 4ti2} \cite{hemmecke20084ti2}, which features
algorithms specifically designed for computing Gr\"obner bases of toric
ideals.  The monomial order on $Y$ is lexicographic, where variables are
ordered by $y_{i,j} < y_{i',j'}$ if $i < i',$ or $i = i'$ and $j < j'$.

\begin{figure}
\[
\begin{array}[t]{cc}
\begin{array}{c}
\framebox{degree 3}\\
\Markov{y_{1,2}y_{0,1}^2-y_{1,0}^2y_{0,2}}\\
\Markov{y_{2,0}y_{0,1}^2-y_{1,0}y_{0,2}^2}\\
y_{2,1}y_{0,2}^2-y_{2,0}^2y_{0,1}\\
\Markov{y_{2,1}y_{1,0}y_{0,2}-y_{2,0}y_{1,2}y_{0,1}}\\
y_{2,1}y_{1,0}^2-y_{1,2}^2y_{0,1}\\
y_{2,1}^2y_{0,2}-y_{2,0}^2y_{1,2}\\
y_{2,1}^2y_{1,0}-y_{2,0}y_{1,2}^2\\
y_{2,1}y_{1,0}y_{0,3}-y_{2,0}y_{1,3}y_{0,1}\\
y_{2,1}^2y_{0,3}-y_{2,0}^2y_{1,3}\\
y_{2,3}y_{1,2}y_{0,2}-y_{2,0}^2y_{1,3}\\
y_{3,0}y_{1,2}y_{0,2}-y_{2,0}y_{1,3}y_{0,3}\\
y_{3,0}y_{1,2}^2-y_{2,0}y_{1,3}^2\\
y_{3,0}y_{2,1}^2-y_{2,3}y_{2,0}y_{1,3}\\
y_{3,1}y_{0,2}^2-y_{2,1}y_{0,3}^2\\
y_{3,1}y_{1,0}y_{0,2}-y_{3,0}y_{1,2}y_{0,1}\\
y_{3,1}y_{1,2}y_{0,2}-y_{2,1}y_{1,3}y_{0,3}\\
y_{3,1}y_{2,3}y_{0,3}-y_{3,0}^2y_{2,1}\\
y_{3,1}^2y_{0,2}-y_{3,0}^2y_{1,2}\\
y_{3,2}y_{1,3}y_{0,3}-y_{3,0}^2y_{1,2}\\
y_{3,2}y_{2,0}y_{1,3}-y_{3,0}y_{2,3}y_{1,2}\\
y_{3,2}y_{2,0}y_{1,4}-y_{3,0}y_{2,4}y_{1,2}\\
y_{3,2}y_{2,1}y_{0,3}-y_{3,1}y_{2,3}y_{0,2}\\
y_{3,2}y_{2,1}y_{0,4}-y_{3,1}y_{2,4}y_{0,2}\\
y_{4,0}y_{2,3}y_{1,3}-y_{3,0}y_{2,4}y_{1,4}\\
y_{4,1}y_{2,3}y_{0,3}-y_{3,1}y_{2,4}y_{0,4}\\
y_{4,2}y_{1,3}y_{0,3}-y_{3,2}y_{1,4}y_{0,4}\\
y_{4,2}y_{2,0}y_{1,3}-y_{4,0}y_{2,3}y_{1,2}\\
y_{4,2}y_{2,1}y_{0,3}-y_{4,1}y_{2,3}y_{0,2}
\end{array}&
\begin{array}{c}
\framebox{degree 2}\\
\Markov{y_{1,3}y_{0,2}-y_{1,2}y_{0,3}}\\
\Markov{y_{2,0}y_{1,0}-y_{1,2}y_{0,2}}\\
y_{2,1}y_{0,1}-y_{1,2}y_{0,2}\\
y_{2,3}y_{0,1}-y_{2,1}y_{0,3}\\
y_{2,3}y_{1,0}-y_{2,0}y_{1,3}\\
y_{3,1}y_{2,0}-y_{3,0}y_{2,1}\\
y_{3,2}y_{0,1}-y_{3,1}y_{0,2}\\
y_{3,2}y_{1,0}-y_{3,0}y_{1,2}\\
\\
\framebox{degree 4}\\
y_{2,1}y_{1,2}y_{0,3}y_{0,2}-y_{2,0}^2y_{1,3}y_{0,1}\\
y_{3,1}y_{2,3}y_{1,3}y_{0,4}-y_{3,0}^2y_{2,1}y_{1,4}\\
y_{3,1}y_{2,3}^2y_{0,4}-y_{3,0}^2y_{2,4}y_{2,1}\\
y_{3,2}y_{2,3}y_{1,3}y_{0,4}-y_{3,0}^2y_{2,4}y_{1,2}\\
y_{4,1}y_{2,3}y_{1,4}y_{0,4}-y_{4,0}^2y_{2,1}y_{1,3}\\
y_{4,1}y_{3,2}y_{1,4}y_{0,4}-y_{4,0}^2y_{3,1}y_{1,2}\\
y_{4,1}y_{3,4}y_{2,4}y_{0,5}-y_{4,0}^2y_{3,1}y_{2,5}\\
\\
\framebox{degree 5}\\
y_{2,1}y_{1,2}^2y_{0,3}^2-y_{2,0}^2y_{1,3}^2y_{0,1}\\
y_{2,1}y_{1,2}^2y_{0,4}y_{0,3}-y_{2,0}^2y_{1,4}y_{1,3}y_{0,1}\\
y_{3,2}y_{2,3}^2y_{1,4}y_{0,4}-y_{3,0}^2y_{2,4}^2y_{1,2}\\
y_{3,2}y_{2,3}^2y_{1,4}y_{0,5}-y_{3,0}^2y_{2,5}y_{2,4}y_{1,2}\\
y_{4,1}y_{2,3}y_{1,4}^2y_{0,5}-y_{4,0}^2y_{2,1}y_{1,5}y_{1,3}\\
y_{4,1}y_{3,2}y_{1,4}^2y_{0,5}-y_{4,0}^2y_{3,1}y_{1,5}y_{1,2}\\
y_{4,3}y_{4,0}^2y_{3,2}y_{3,1}-y_{4,2}y_{4,1}y_{3,4}^2y_{0,3}\\
y_{5,1}y_{4,2}y_{3,5}^2y_{0,3}-y_{5,0}^2y_{4,3}y_{3,2}y_{3,1}
\end{array}
\end{array}
\]
\caption{An $\Inc(\NN)$-equivariant Gr\"obner basis for $\ker \omega$ in Example~\ref{ex:Sym21}. 
The five highlighted binomials form a $\Sym(\NN)$-equivariant Markov basis according to~\cite{KKL:equivariant-markov}.}
\label{figure:Sym21}
\end{figure}

The result displayed in Figure~\ref{figure:Sym21} consists of 51 generators
with indices from $\{0,1,2,3,4,5\}$ and degrees up to 5.
Note that a minimal generating set resulting from a study of the family of equivariant toric maps of the form 
\[y_{ij}\mapsto z_i^az_j^b,\quad i,j\in\NN, i\neq j,\]
for fixed $a,b\in\NN$ in~\cite{KKL:equivariant-markov} is much smaller.
\end{ex}

\begin{re}\label{re:Inc12}
As pointed out in the Introduction, the technique laid out in this article
does not settle the question whether the finite generatedness of $\ker
\phi$ in the Main Theorem persists when $\Inc(\NN)$ acts with finitely
many orbits on $Y$ and the monomial map $\phi$ is required to be merely
$\Inc(\NN)$-equivariant (though we do know that $\im \phi$ needs not be
$\Inc(\NN)$-Noetherian in this case).

However, a na\"ive elimination procedure terminates, for instance, for
the $\IncNN$-analogue of Example~\ref{ex:Sym21}, i.e., for the same map,
but with the smaller set of variables
\[Y := \{y_{j_0,j_1} \mid j_0,j_1 \in \NN,\; j_0 \boldmath{>} j_1\}.\]
A computation that can be carried out with {\em EquivariantGB}~\cite{EquivariantGB} produces a finite number of generators of the kernel:
\[\{y_{3,1}y_{2,0} - y_{3,0}y_{2,1},\; y_{3,2}^2y_{1,0} - y_{3,1}y_{3,0}y_{2,1},\; y_{4,2}y_{3,2}y_{1,0} - y_{4,0}y_{3,1}y_{2,1}\}.\]
\end{re}

%
%


\bibliographystyle{alpha}
\bibliography{inftoric}

\end{document}